\documentclass[11pt]{amsart}
\usepackage[margin=1.00in]{geometry}
\usepackage{amscd,amssymb, amsmath, wasysym, yfonts}
\usepackage{graphicx}
\usepackage[all]{xy}
\usepackage{url}
\usepackage{hyperref}
\hypersetup{
 colorlinks=true,
 citecolor=cyan,
 linkcolor=blue,
 urlcolor=cyan}

\theoremstyle{plain}
\newtheorem{theorem}{Theorem}[section]

\newtheorem{prop}[theorem]{Proposition}
\newtheorem{criterion}[theorem]{Criterion}

\newtheorem{claim}[theorem]{Claim}
\newtheorem{cor}[theorem]{Corollary}
\newtheorem{conj}[theorem]{Conjecture}

\theoremstyle{definition}

\newtheorem{rmk}[theorem]{Remark}

\newtheorem*{ex*}{Example}
\newtheorem{problem}[theorem]{Problem}

\newcommand\sO{{\mathcal O}}
\newcommand\sK{{\mathcal K}}

\newcommand\sC{{\mathcal C}}

\newcommand\sF{{\mathcal F}}

\newcommand\sI{{\mathcal I}}

\newcommand\sJ{{\mathcal J}}
\newcommand\sQ{{\mathcal Q}}
\newcommand\sA{{\mathcal A}}

\newcommand{\codim}{{\rm codim}\,}

\newcommand\rp{{\mathbf{P}}}
\newcommand\rz{{\mathbf{Z}}}
\newcommand\rr{{\mathbf{R}}}
\newcommand\rc{{\mathbf{C}}}

\newcommand\rl{{\mathbf{L}}}

\title[Deformations of minimal cohomology classes]
{Deformations of minimal cohomology classes on abelian varieties}

\author{Luigi Lombardi and Sofia Tirabassi}
\address{Mathematical Institute\\ University of Bonn, Endenicher Allee 60, 53115, Germany}
 \email{\url{lombardi@math.uni-bonn.de}}

\address{Department of Mathematics\\ The University of Utah, 155 S 1400 E, JWB 233, Salt Lake City, UT 84112-0090, USA}
\email{\url{sofia@math.utah.edu}}

\begin{document}

\begin{abstract}
We show that the infinitesimal deformations of the Brill--Noether locus $W_d$ attached to a smooth non-hyperelliptic 
curve $C$ are in one-to-one correspondence with the deformations 
of $C$. As an application, we prove that if a Jacobian $J$ deforms together with a minimal cohomology class
out the Jacobian locus, then $J$ is hyperelliptic. In particular, this provides an evidence to a conjecture of Debarre on the classification
of ppavs carrying a minimal cohomology class. Finally, we also study simultaneous deformations of Fano surfaces of lines and intermediate Jacobians.
 \end{abstract}

\maketitle

\section{Introduction}

Given a smooth complex curve $C$ of genus $g\geq 2$ with Jacobian $J$, 
we denote by $C_d$ ($d\geq 1$) the $d$-fold symmetric product of $C$ and by 
$$f_d:C_d\longrightarrow J,\quad \quad P_1+\ldots +P_d\mapsto \sO_C(P_1+\ldots +P_d-dQ)$$
the Abel--Jacobi map defined up to the choice of a point $Q\in C$.
In the papers \cite{Ke2} and \cite{Fa} the infinitesimal deformations of $C_d$ and $f_d$ are studied: these
are in one-to-one correspondence with the deformations of $C$ if and only if $g\geq 3$.
In particular, there are isomorphisms of functors of Artin rings 
$${\rm Def}_{C_d}\simeq {\rm Def}_{f_d}\simeq {\rm Def}_C\quad\quad  \mbox{ for\, all }\quad \quad d\geq 2 \quad \quad \mbox{ if \,and\, only\, if }\quad 
\quad g\geq 3.$$
On the other hand, the computation of infinitesimal deformations of the images 
$$W_d=\{[L]\in {\rm Pic}^d(C) \, | \, h^0(C,L)>0\}$$ 
of Abel--Jacobi maps, namely the Brill--Noether loci parameterizing degree $d$ line bundles on $C$ having at least one 
non-zero global section, is a problem that has not been studied yet in its full generality and has
interesting relationships to a conjecture of Debarre (see Conjecture \ref{conjintr}).
Previous calculations of deformations of Brill--Noether loci have been performed only for Theta divisors $\Theta\simeq W_{g-1}$ of 
non-hyperelliptic Jacobians where the authors of \cite{SV} prove that the
first-order deformations of $C$ inject in those of $W_{g-1}$.

One of the main difficulties for the computation of deformations of Brill--Noether loci
is that in general these spaces are singular. However, as shown by work 
of Kempf \cite{Ke1}, the singularities of $W_d$ are at most rational and a resolution of its singularities is provided by the 
Abel--Jacobi map $u_d:C_d\rightarrow W_d$ which factorizes $f_d$. 
By extending a construction of Wahl in \cite[\S1]{Wa} for affine equisingular deformations,
this allows us to define a ``blowing-down deformation'' morphism of functors of Artin rings
$$u'_d:{\rm Def}_{C_d}\longrightarrow {\rm Def}_{W_d}$$ 
sending an infinitesimal deformation $\mathcal{C}_d$ of $C_d$ over an Artinian local $\rc$-algebra $A$, 
to the infinitesimal deformation 
$\mathcal{W}_d:=\big(W_d,u_{d*}\sO_{\mathcal{C}_d}\big)$ of $W_d$ over $A$ (Proposition \ref{bdmor}). 
In the following theorem we prove that the blowing-down morphism is an isomorphism of functors whenever $C$ is non-hyperelliptic. 
We refer to \S\ref{defWd} for its proof.
\begin{theorem}\label{intrdefW_d}
If $C$ is a smooth non-hyperelliptic curve of genus $g\geq 3$, then for all $1\leq d < g-1$ the blowing-down morphism 
$u'_d:{\rm Def}_{C_d}\rightarrow {\rm Def}_{W_d}$ is an isomorphism. In particular, ${\rm Def}_{W_d}\simeq {\rm Def}_C$,
$W_d$ is unobstructed, and ${\rm Def}_{W_d}$ is prorepresented by a formal power series in $3g-3$ variables.
\end{theorem}

We believe that the hypothesis of non-hyperellipticity is unnecessary at least in cases $d\neq 2$.
However, for $d=2$, we notice that the space $W_2$, besides the deformations coming from the curve, may also acquire additional 
deformations coming from 
deformations of Fano surfaces of lines associated to smooth cubic threefolds. In fact, Collino in \cite{Co} 
shows that the locus of hyperelliptic Jacobians of dimension five lies in the 
boundary of the locus of intermediate Jacobians of smooth cubic threefolds.
  
Now we make some comments regarding the proof of Theorem \ref{intrdefW_d}. 
In case the exceptional locus of $u_d$ has codimension at least three in $C_d$, then Theorem \ref{intrdefW_d} easily follows 
from the general theory of blowing-down morphisms (see \S\ref{secbdm} and in particular Criterion \ref{cri} below) and does not rely on the special 
structure of Abel--Jacobi maps.
In fact,  
the existence of blowing-down morphisms is not specific to the morphism $u_d$ itself, 
rather to any morphism $f:X\rightarrow Y$ between projective integral schemes such that $\rr f_*\sO_X\simeq \sO_Y$
(Proposition \ref{bd}). In \S\ref{secbdm}, we give 
an explicit description to the differential of a blowing-down morphism $f':{\rm Def}_{X}\rightarrow {\rm Def}_Y$ and, moreover, 
we find sufficient conditions on $X$, $Y$ and $f$ so that $f'$ defines an isomorphism of functors. This leads
to the following criterion whose proof can be found in Corollary \ref{corbd}. We refer to \cite{Ran3} for related criteria regarding 
source-target-stability type problems.
 \begin{criterion}\label{cri}
  Let $f:X\rightarrow Y$ be a birational morphism of integral projective schemes over an algebraically closed field of characteristic zero such
  that the exceptional locus of $f$ is of codimension at least 
 three in $X$. Furthermore assume that $\rr f_*\sO_X\simeq \sO_Y$. If $X$ is non-singular, unobstructed and  
 $h^0(X,T_X)=0$, then the blowing-down morphism
 $f':{\rm Def}_X\rightarrow {\rm Def}_Y$ is an isomorphism of functors of Artin rings.
\end{criterion}

Hence the difficult case of Theorem \ref{intrdefW_d} is precisely when 
the exceptional locus of $u_d$ is of codimension two (the case of codimension one is excluded as we are supposing $C$ 
non-hyperelliptic). 
For this case we carry out an ad-hoc argument specific
to Abel--Jacobi maps. The main point is to prove that the differential ${\rm d}u_d'$ of the blowing-down morphism is an isomorphism even in this case. 
To this end, first of all we notice that the kernel and cokernel of ${\rm d}u_d'$ are identified to the groups
${\rm Ext}^1_{\sO_{C_d}}(P,\sO_{C_d})$ and ${\rm Ext}^2_{\sO_{C_d}}(P,\sO_{C_d})$, respectively, 
where $P$ is the cone of the following composition of morphisms of complexes:
\begin{equation}\label{compintr}
\rl u_{d}^*\Omega_{W_d}\longrightarrow u_d^*\Omega_{W_d}\longrightarrow \Omega_{C_d}
\end{equation}
where the first map is the truncation morphism and the second is the natural morphism between sheaves of K\"{a}hler differentials.
On the other hand, an  application of Grothendieck--Verdier's duality shows that  
the vanishings of the above mentioned Ext-groups hold as soon as the support of the higher direct image 
sheaf $R^1 f_{d*}(\Omega_{C_d / W_d}\otimes \omega_{C_d})$ has
sufficiently high codimension in $J$, namely at least five (see Propositions \ref{R1a} and \ref{relativediff}). 
But this is ensured by Ein's computations of the Castelnuovo--Mumford regularity of 
the dual of the normal bundle to the fibers of $f_d$ 
(\cite{Ein}). Finally, the passage from first-order deformations to arbitrary infinitesimal deformations follows as $C_d$ is unobstructed.

As an application, we compute the infinitesimal deformations of the Albanese map
$$\iota_d:W_d\hookrightarrow J,\quad \quad L\mapsto L\otimes \sO_C(-dQ)$$ where both the domain, the codomain, and the closed 
immersion are allowed to deform (Sernesi in \cite[Example 3.4.24 (iii)]{Se} solves the case $d=1$).
The importance of this problem relies 
on a conjecture of Debarre pointing to a classification of $d$-codimensional subvarieties $X$ of ppavs $(A,\Theta)$ representing a minimal cohomology class, \emph{i.e.} 
$[X]= \frac{1}{d!}[\Theta]^d$ in $H^{2d}(A,\rz)$ (\emph{cf}. \cite{De1}, \cite{De2} and \cite{G}).
\begin{conj}\label{conjintr}
 Let $(A,\Theta)$ be an indecomposable ppav of dimension $g$ and let $X\subset A$ be a reduced equidimensional
  subscheme of codimension $d>1$.
  Then $[X]=\frac{1}{d!}[\Theta]^d$ in $H^{2d}(A,\rz)$ if and only if,
 up to isomorphism, either $(A,\Theta)$ is the Jacobian of a smooth curve and $X=\pm W_d$, or $(A,\Theta)$ is an intermediate Jacobian of a smooth 
 cubic threefold $Y$ and $X=\pm F$ where $F$ is the Fano surface parameterizing lines on $Y$.
   \end{conj}

Debarre establishes the previous conjecture in the case of 
Jacobians of smooth curves by proving that the only effective cycles 
on Jacobians representing the minimal class $\frac{1}{d!}[\Theta]^d$ are the loci $\pm W_d$ up to translation (\cite{De1} Theorem 5.1). 
Moreover, Debarre himself solves the conjecture in a weak sense 
by proving that the Jacobian locus $\sJ_g$ in the moduli space of ppavs $\sA_g$ 
is an irreducible component of the locus $\sC_{g,d}$ of $g$-dimensional ppavs carrying an effective 
cycle representing the minimal class $\frac{1}{d!}[\Theta]^d$, 
and similarly for the locus of intermediate Jacobians (\cite[Theorem 8.1]{De1}). (Other evidence towards Conjecture \ref{conjintr} 
can be found in \cite[Theorem 5]{Ran1} for ppavs of dimension four and in \cite[Theorem 1.2]{Ho} for generic intermediate Jacobians.)
Therefore the study of deformations of type 
 
 \begin{eqnarray*}
\xymatrix@=22pt{
 \mathcal{W}_d \ar[rd] \ar[rr] && \mathcal{J}\ar[ld]\\
 & {\rm Spec}\, A\\
}
\end{eqnarray*}

\noindent over an Artinian local $\rc$-algebra $A$ such that the restriction to the closed point is the closed embedding
$\iota_d:W_d\hookrightarrow J$ will tell us in which directions the Jacobian $J$ is allowed to deform as a ppav containing a subvariety 
representing a minimal class. More precisely, this study will suggest us
along which type of Jacobians there might be an irreducible component of $\sC_{g,d}$ (different from $\sJ_g$) that passes through them.
The main result of this paper in this direction is an evidence to Conjecture \ref{conjintr} mainly saying that non-hyperelliptic 
Jacobians, seen as elements in $\mathcal{C}_{g,d}$, deform along the expected directions.
Less informally, if 
$$p_{W_d}:{\rm Def}_{\iota_d}\longrightarrow {\rm Def}_{W_d}$$ denotes the natural forgetful morphism, we have then:

\begin{theorem}\label{intrdefiota}
 If $C$ is a smooth non-hyperelliptic curve of genus $g\geq 3$, then for any $1\leq d<g-1$ the forgetful morphism
 $p_{W_d}:{\rm Def}_{\iota_d}\rightarrow {\rm Def}_{W_d	}$ is an isomorphism. In particular, 
 $\iota_d$ is unobstructed and ${\rm Def}_{\iota_d}$ is prorepresented by a formal power series in $3g-3$ variables.
\end{theorem}

Combining with Theorem \ref{intrdefW_d} we obtain:

\begin{cor}\label{intrcor}
 Under the hypotheses of Theorem \ref{intrdefiota} there exists an isomorphism of functors of Artin rings 
 ${\rm Def}_{\iota_d}\simeq {\rm Def}_C$ for every $1\leq d < g-1$. Hence, if $J$ is a non-hyperelliptic Jacobian, 
 then any infinitesimal deformation of $J$ together with an infinitesimal deformation of  
 the minimal class $W_d$ deforms $J$ along the Jacobian locus.
\end{cor}

The proof of Theorem \ref{intrdefiota} still relies on the use of blowing-down morphisms.
More precisely, we prove that not only deformations of schemes with rational singularities can be blown-down, 
but also deformations of morphisms between them (Proposition \ref{bdmor}). 
Thus there is a well-defined morphism of functors 
$$F: {\rm Def}_{f_d}\longrightarrow {\rm Def}_{\iota_d}$$ 
which we prove to be an isomorphism, by means of Theorem \ref{intrdefW_d} and Ran's formalism of
deformations of morphisms recalled in details in \S\ref{secdefmaps}. 
Finally, as by work of Kempf \cite{Ke2} the forgetful morphism ${\rm Def}_{f_d}\rightarrow {\rm Def}_{C_d}$ is an isomorphism, 
we deduce that so is $p_{W_d}$.
\begin{problem}
 
Pareschi--Popa in \cite[Conjecture A]{PP} suggest that $d$-dimensional subvarieties $X$ of $g$-dimensional ppavs $(A,\Theta)$ 
representing a minimal cohomology class should be characterized as those for which 
the twisted ideal sheaf $\sI_X(\Theta)$ is $GV$ (we recall that a sheaf $\sF$ on an abelian variety $A$ is $GV$ if 
$\codim V^i(\sF)\geq i$ for all $i>0$ where $V^i(\sF):=\{\alpha \in {\rm Pic}^0(A) \, | \, h^i(A,\sF\otimes \alpha)>0\}$). 
In fact, one of the main results of \cite{PP} is that the $GV$ condition on $X$ implies that $X$ has minimal class.
It would be interesting to check whether this property is stable under infinitesimal deformations in order to 
get information concerning the geometry of the corresponding loci in $\sA_g$ for all $d$. Steps in this direction are again due to Pareschi--Popa 
as they prove, without appealing to deformation theory but using the technique of Fourier--Mukai transforms, that for $d=1$ and $d=g-2$ 
this locus coincide precisely with the Jacobian locus in $\sA_g$ (\cite[Theorem C]{PP}).

\end{problem}

\subsection*{Notation} 
In these notes a \emph{scheme} is a  
separated scheme of finite type defined over an algebraically closed field $k$ of characteristic zero, unless otherwise specified.  
We denote by $\Omega_X$ the sheaf of K\"{a}hler differentials and by $T_X$ the tangent sheaf of a scheme $X$.

\subsection*{Acknowledgments} Both authors are grateful to Andreas H\"{o}ring for conversations regarding Conjecture \ref{conjintr}.
LL is grateful to Daniel Huybrechts who introduced him to the theory 
of deformations and for discussions that inspired the results of this work. 
Moreover he is thankful to Daniel Greb, 
S\'{a}ndor Kov\'{a}cs, Andreas Krug, Mihnea Popa, Taro Sano, Christian Schnell, Stefan Schreieder, Lei Song and Zhiyu Tian for answering to all his questions. 
Finally, special thanks go to Alberto Bellardini and Mattia Talpo for correspondences and clarifications on moduli spaces.
ST is grateful to Herb Clemens for having explained her a lot of deformation theory and for the patience with which he answered all her questions. 
She is also thankful to Christopher Hacon and Elham Izadi for very inspiring mathematical conversations. 

\subsection*{Support}
This collaboration started when ST visited the University of Illinois at Chicago 
supported by the AWM-NSF Mentoring Grant (NSF award number DMS-0839954), while LL was a graduate student there. 
Both authors are grateful to UIC for the nice working environment and the kind hospitality. 
LL was supported by the SFB/TR45 ``Periods, moduli spaces, and arithmetic of algebraic varieties'' of the DFG (German Research Foundation).

\section{Deformations of morphisms}\label{secdefmaps}

We begin by recalling Ran's theory on deformations of morphisms between compact complex spaces extending works of Horikawa in the smooth case
(\emph{cf}. \cite{Ran2,Ran3,Ran4}). 
The deformations considered by Ran allow to deform both the domain and the codomain of a morphism.
For the purposes of this work we present Ran's theory for 
the category of schemes defined over an algebraically closed field $k$ of 
characteristic zero.

Let $X$ be a reduced projective $k$-scheme. We define the spaces
$$T^i_X:={\rm Ext}_{\sO_X}^i(\Omega_X,\sO_X)$$ and 
denote by ${\rm Def}_X$ the functor of Artin rings of deformations of $X$ up to isomorphism. 
We recall that $T_X^1$ is the tangent space to ${\rm Def}_X$. Under this identification
a first-order deformation $\pi: \mathcal{X}\rightarrow {\rm Spec}\,\big(k[t]/(t^2)\big)$ of $X$ is sent 
to the extension class determined by the conormal sequence of the closed immersion
$X\subset \mathcal{X}$: 
 $$0\longrightarrow \sO_X\longrightarrow \Omega_{\mathcal{X}|X}\longrightarrow \Omega_X\longrightarrow 0$$
 (the fact that $\pi$ is flat implies that $\sO_X$ is the conormal bundle of $X\subset \mathcal{X}$, while the fact that 
 $X$ is reduced implies that the conormal sequence is exact also on the left).
Moreover, if $X$ is a locally complete intersection, then $T_X^2$ is an obstruction space 
(\emph{cf}. \cite[Theorem 2.4.1 and Proposition 2.4.8]{Se}).

 Let $Y$ be another reduced projective $k$-scheme, and let $f:X\rightarrow Y$ be a morphism. 
 A \emph{deformation of} $f:X\rightarrow Y$ over an Artinian local $k$-algebra $A$ with residue field $k$ is a diagram of commutative squares 
 and triangles
 
   \begin{eqnarray}\label{definition}
\xymatrix@=32pt{
  X \ar[dd]\ar[rd]^f\ar[rr] && \mathcal{X} \ar'[d][dd]^g\ar[rd]^{\tilde f} \\
 & Y \ar[rr]\ar[ld] && \mathcal{Y}\ar[ld]^h \\ 
 {\rm Spec\,} k \ar[rr] && {\rm Spec}\, A \\}
 \end{eqnarray}
 
 \noindent such that $\mathcal{X}$ and $\mathcal{Y}$ are deformations of $X$ and $Y$ over $A$ respectively  
 and $\tilde f$ restricts to $f$ when pulling-back to ${\rm Spec\,}k$. 
 We say that a deformation $\widetilde f:\mathcal{X}\rightarrow \mathcal{Y}$ of $f$ is isomorphic to another deformation $\widehat{f}:\mathcal{X}'
 \rightarrow \mathcal{Y}'$ of $f$ if there are isomorphisms $\varphi:\mathcal{X}\rightarrow \mathcal{X'}$ and $\phi:\mathcal{Y}
 \rightarrow \mathcal{Y}'$ over ${\rm Spec}\, A$ such that $\widehat{f}\circ \varphi = \phi \circ \widetilde{f}$ and $j^{-1} \circ (\varphi \times _{k}A)\circ i = 
 {\rm id}_X$ where $i:X\rightarrow \mathcal{X}\times _k A$ and $j:X'\rightarrow \mathcal{X}'\times _k A$ are the isomorphisms determined by 
 $\mathcal{X}$ and $\mathcal{X}'$ respectively, and similarly for $\mathcal{Y}$, $\mathcal{Y}'$, and $\phi$.
 We denote by ${\rm Def}_f$ the functor of Artin rings of deformations of $f$ up to isomorphism.
The functors ${\rm Def}_X$ and ${\rm Def}_f$ satisfy Schlessinger's conditions $(H_0)$, $(H_1)$, $(H_2)$ and $(H_3)$ when both $X$ and $Y$ are 
projective schemes (\emph{cf}. \cite{S}).

\subsection{Tangent space}

 One of the central results in \cite{Ran2} is that the first-order
 deformations of a morphism $f:X\rightarrow Y$ are controlled by a certain space $T_f^1$ defined as follows. Let
 \begin{eqnarray}\label{delta}
 \delta_0: \sO_Y\longrightarrow f_*\sO_X \quad \mbox{ and }\quad 
 \delta_{1}:f^*\Omega_Y\longrightarrow \Omega_X
 \end{eqnarray}
 be the natural morphisms induced by $f$ and let 
 $${\rm ad}(\delta_0): f^*\sO_Y\longrightarrow \sO_X$$ be the morphism induced by $\delta_0$ via adjunction.
 Then we define $T^1_f$ to be the abelian group consisting 
 of isomorphism classes of triples $(e_X,e_Y,\gamma)$ such that 
$e_X$ and $e_Y$ are classes in ${\rm Ext}^1_{\sO_X}(\Omega_X,\sO_X)$ and 
	${\rm Ext}^1_{\sO_Y}(\Omega_Y,\sO_Y)$ determined by the conormal sequences of some deformations $\mathcal{X}$ and $\mathcal{Y}$ 
	of $X$ and $Y$ respectively:
	\begin{eqnarray}\label{ex}
	e_X \,: \,0\rightarrow \sO_X\rightarrow \Omega_{\mathcal{X}|X} \rightarrow \Omega_X\rightarrow 0\,\quad \mbox{ and }\,\quad 
	e_Y\, : \, 0\rightarrow \sO_Y\rightarrow \Omega_{\mathcal{Y}|Y} \rightarrow \Omega_Y\rightarrow 0;
	\end{eqnarray}
  and $\gamma:f^*\Omega_{\mathcal{Y}|Y}\rightarrow \Omega_{\mathcal{X}|X}$ is a morphism such that the following diagram
  
\begin{eqnarray}\label{diagrT1a}
\xymatrix@=32pt{
 f^*\sO_Y \ar[r]\ar[d]^{{\rm ad} (\delta_0)} & f^*\Omega_{\mathcal{Y}|Y} \ar[r]\ar[d]^{\gamma}  
& f^*\Omega_Y \ar[d]^{\delta_1}\\
 \sO_X \ar[r]  & \Omega_{\mathcal{X}|X}    \ar[r] & \Omega_X \\}
   \end{eqnarray}

 \noindent commutes. For future reference we recall \cite[Proposition 3.1]{Ran2} revealing the role of $T_f^1$.
\begin{prop}\label{tangentspacedef}
Let $X$ and $Y$ be projective reduced $k$-schemes and let $f:X\rightarrow Y$ be a morphism. Then $T_f^1$ is the tangent space to ${\rm Def}_f$.
\end{prop}
\begin{proof}
As already pointed out earlier, the datum of two extension classes $e_X\in T_X^1$ and $e_Y\in T^1_Y$ is equivalent to giving two Cartesian diagrams

\begin{eqnarray*}
  \centerline{ \xymatrix@=32pt{
 X \ar[r]\ar[d] & \mathcal{X} \ar[d]^{g} &  Y\ar[d] \ar[r]& \mathcal{Y} \ar[d]^{h}\\
 {\rm Spec}\, k \ar[r] & {\rm Spec}\, \big(k[t]/(t ^2)\big) & {\rm Spec}\, k \ar[r] & {\rm Spec}\, \big(k[t]/(t^2)\big) \\}}
\end{eqnarray*}

\noindent such that both $g$ and $h$ are flat morphisms. 
On the other hand, as it is shown in \cite[Theorem 1.6]{BE}, the existence of a morphism $\tilde f:\mathcal{X}\rightarrow\mathcal{Y}$ such that the 
top square of \eqref{definition} commutes is equivalent to the existence of a morphism $\gamma: f^*\Omega_{\mathcal{Y}|Y}\rightarrow 
\Omega_{\mathcal{X}|X}$ such that the right square of \eqref{diagrT1a} commutes.
Finally, it is not hard to prove that the commutativity of the rightmost triangle in \eqref{definition} is equivalent to the commutativity 
of the left 
square of \eqref{diagrT1a}.
\end{proof}

\subsection{Ran's exact sequence}

In order to study the space $T_f^1$ usually one appeals to an exact sequence relating $T_f^1$ to the tangent spaces 
$T_X^1$ and $T_Y^1$. This sequence turns out to be extremely useful to 
study stability and co-stability properties of a morphism, and furthermore, in some situations, it suffices to determine the group 
$T_f^1$ itself (\emph{cf}. \cite{Ran3}).

Let $T^0_f$ be the group consisting of pairs of morphisms 

\begin{eqnarray}\label{morphisms}
 a:\Omega_X\longrightarrow \sO_{X}\quad \mbox{ and }\quad b:\Omega_Y\longrightarrow 
 \sO_Y
  \end{eqnarray}

  such that the following diagram 
  
  \begin{eqnarray*}
\centerline{ \xymatrix@=32pt{
 f^*\Omega_Y \ar[r]^{f^*b}\ar[d]^{\delta_1} & f^*\sO_Y\ar[d]^{{\rm ad}(\delta_0)} \\
 \Omega_X \ar[r]^{a} & \sO_{X} \\}} 
  \end{eqnarray*}   
  
 \noindent commutes, and consider the sequence 
\begin{gather}\label{exactseq}
0\longrightarrow T_f^0\longrightarrow T_X^0\oplus T_Y^0\stackrel{\lambda_0}{\longrightarrow} 
{\rm Hom}_{\sO_{X}}(f^* \Omega_Y,\sO_X)\stackrel{\lambda_1}{\longrightarrow}
T_f^1\longrightarrow T_X^1\oplus T^1_Y\stackrel{\lambda_2}{\longrightarrow} {\rm Ext}^1_{\sO_{X}}(\rl f^*\Omega_Y,\sO_X),
\end{gather}
where the morphisms without names are the obvious ones, while the others are defined as follows. 
 Given a pair of morphisms as in \eqref{morphisms}, we set
 $\lambda_0 (a,b)= a\circ \delta_1-{\rm ad}(\delta_0)\circ f^*b$.
On the other hand, $\lambda_1$ takes a morphism $\varepsilon:f^*\Omega_Y\rightarrow \sO_X$ to
the trivial extensions in $T^1_X$ and $T^1_Y$ together with the morphism  
$\delta= \left( \begin{smallmatrix} {\rm ad}(\delta_0) &\varepsilon \\ 0 & \delta_1 \end{smallmatrix} \right)$  
so that the following diagram 

\begin{eqnarray*}
\centerline{ \xymatrix@=32pt{
f^*\sO_Y \ar[r]\ar[d]^{{\rm ad}(\delta_0)} & f^*\sO_{Y}\oplus f^*\Omega_Y \ar[r]\ar[d]^{\delta}  & f^*\Omega_Y \ar[d]^{\delta_1}\\
\sO_X \ar[r] & \sO_X\oplus \Omega_X \ar[r] & \Omega_X \\}} 
 \end{eqnarray*}

  \noindent commutes.
Finally, in order to define $\lambda_2$, we introduce some additional notation.
Let 
\begin{eqnarray}\label{truncation}
\xi:f_*\sO_X\longrightarrow \rr f_*\sO_X\quad \mbox{ and }\quad \zeta:\rl f^*\Omega_Y\longrightarrow f^*\Omega_Y 
\end{eqnarray}
be the truncation morphisms
 (see \cite[Exercise 2.32]{Huy}) and set
 \begin{eqnarray}\label{deltabar}
 \bar \delta_0:=\xi \circ \delta_0 :\sO_Y\longrightarrow \rr f_*\sO_X\quad \mbox{ and }\quad 
 \bar \delta_1:=\delta_1 \circ \zeta :\rl f^*\Omega_Y\longrightarrow \Omega_X.
 \end{eqnarray}
 We denote by $\lambda_2^1$ and $\lambda_2^2$ the components of $\lambda_2$ and we set 
 $\lambda_2(e_X,e_Y)=\lambda_2^1(e_X)-\lambda_2^2(e_Y)$ where $e_X$ and $e_Y$ are extension classes as in \eqref{ex}, 
so we only need to define $\lambda_2^1$ and $\lambda_2^2$.
Thinking of the extensions $e_X\in T^1_X$ and $e_Y\in T_Y^1$ as morphisms
of complexes 
$$\alpha:\Omega_X\longrightarrow \sO_X[1]\quad \mbox{ and }\quad \beta:\Omega_Y\longrightarrow \sO_Y[1],$$ 
we then set
$$\lambda_2^1(\alpha)=\alpha\circ \bar \delta_1 \quad \mbox{ and }\quad \lambda_2^2(\beta)={\rm ad}\,(\bar \delta_0[1]\circ \beta)$$
where ${\rm ad}(-)$ denotes the adjunction isomorphism of \cite[Corollary 5.11]{Ha1}. 
More concretely, we have $\lambda_2^2(\beta)={\rm ad}\,(\bar \delta_0[1] \circ \beta) \simeq \varrho\circ \rl f^*(\bar \delta_0[1]\circ \beta)$
where $\varrho: \rl f^* \rr f_* \sO_X\rightarrow \sO_X$ is the natural morphism induced by adjunction.
\begin{prop}\label{exactprop}
If $f:X\rightarrow Y$ is a morphisms of reduced $k$-schemes such that $f(X)$ is not contained in the singular locus of $Y$,
then the sequence \eqref{exactseq} is exact.
\end{prop}

\begin{proof}
 We show exactness only at the term $T^1_X\oplus T_Y^1$, exactness at the other terms follows easily from the definition of our objects and maps.
 First of all we show that the composition $T^1_f\rightarrow T^1_X\oplus T^1_Y\stackrel{\lambda_2}{\rightarrow} 
 {\rm Ext}_{\sO_X}^1(\rl f^*\Omega_Y,\sO_X)$ is zero. 
 Let $e_X$ and $e_Y$ be two extension classes as in \eqref{ex} which we think of
as morphisms of complexes $\alpha:\Omega_X\rightarrow 
 \sO_X[1]$ and $\beta:\Omega_Y\rightarrow \sO_Y[1]$, 
 and suppose that there exists a morphism 
 $\gamma:f^*\Omega_{\mathcal{Y}|Y}\rightarrow \sO_{\mathcal{X}|X}$
 such that 
 \eqref{diagrT1a} commutes. 
 By our assumption on $f(X)$, together with generic freeness and base change, we see that the sheaf $L^{-1}f^*\Omega_Y$ is torsion on $X$. 
This in particular yields the exactness of 
  the sequence 
  \begin{eqnarray}\label{pullbackseq}
   0\longrightarrow f^*\sO_Y\longrightarrow f^*\Omega_{\mathcal{Y}|Y}\longrightarrow f^*\Omega_Y\longrightarrow 0.
  \end{eqnarray}
 Therefore the commutativity of \eqref{diagrT1a} implies the commutativity of the following diagram of distinguished triangles 
 
  \begin{eqnarray}\label{diagrT1b}
\xymatrix@=32pt{
\rl f^*\sO_Y \ar[r]\ar[d]^{{\rm ad}(\bar \delta_0)} & \rl f^*\Omega_{\mathcal{Y}|Y} \ar[r]\ar[d]^{\bar \gamma}  
& \rl f^*\Omega_Y \ar[d]^{\bar \delta_1} \ar[r]^{\rl f^*\beta} & \rl f^*\sO_Y[1]\ar[d]^{{\rm ad}(\bar \delta_0)[1]}\\
 \sO_X \ar[r]  & \Omega_{\mathcal{X}|X} \ar[r] & \Omega_X \ar[r]^{\alpha} & \sO_X[1] \\}
   \end{eqnarray}
 
 \noindent where $\bar \gamma$ denotes the composition $\rl f^*\Omega_{\mathcal{Y}|Y}\rightarrow f^*\Omega_{\mathcal{Y}|Y}\stackrel{\gamma}{\rightarrow}
   \Omega_{\mathcal{X}|X}$.
  Hence in particular the commutativity of the right-most square tells us that 
\begin{eqnarray*}  
\lambda_2^2(\beta) & = & {\rm ad}(\bar \delta_0[1]\circ \beta)\\
& \simeq & \varrho \circ \rl f^*(\bar \delta_0[1])\circ \rl f^* \beta\\
& \simeq & {\rm ad}(\bar \delta_0[1]) \circ \rl f^*\beta\\
& \simeq & \alpha \circ \bar \delta_1 =  \lambda_2^1(\alpha),
 \end{eqnarray*}
and therefore $\lambda_2(\alpha,\beta)\simeq 0$.

On the other hand, if $\lambda_2(\alpha,\beta)= \lambda_2^1(\alpha)-\lambda_2^2(\beta)\simeq 0$, then 
 $\alpha \circ \bar \delta_1 \simeq {\rm ad}(\bar \delta_0[1] \circ \beta) \simeq {\rm ad}(\bar \delta_0[1]) \circ \rl f^*\beta$ which implies the 
 commutativity of \eqref{diagrT1b}. By taking cohomology in degree $0$ and by the fact that \eqref{pullbackseq} is exact, 
 we conclude that \eqref{diagrT1a} is commutative too. Therefore the triple $(e_X,e_Y,\gamma)$ lies in the image of $\lambda_1$.
  \end{proof}
  
  \begin{rmk}
   The functor ${\rm Def}_f$ comes equipped with two \emph{forgetful morphisms} $p_X:{\rm Def}_f\rightarrow {\rm Def}_X$ and 
   $p_Y:{\rm Def}_f\rightarrow {\rm Def}_Y$ obtained by the fact that a deformation of $f$ determines both 
   a deformation of $X$ and one of $Y$. 
   The differential ${\rm d}p_X$ is identified to the morphism $T_f^1\rightarrow T_X^1$ of the sequence \eqref{exactseq}, and similarly
   for the differential ${\rm d}p_Y$.
  \end{rmk}

  \begin{rmk}
   Obstruction spaces to the functor ${\rm Def}_f$ are studied in \cite{Ran2} in complete generality. 
   However, we are able to follow Ran's argument only under the additional hypothesis that the involved spaces are 
   locally complete intersections. As in this paper we will be dealing with 
   varieties which are not locally complete intersections, we refrain to give a systematic description of the obstructions, but rather we refer 
   to \cite{Ran2} and \cite{Ran4} to get a flavor of this theory.
  \end{rmk}

\subsection{Blowing-down deformations}\label{secbdm}

We recall that a resolution of singularities $X$ of a variety $Y$ 
with at most rational singularities induces a morphism of functors of Artin rings 
${\rm Def}_X\rightarrow {\rm Def}_Y$ (\emph{cf}. \cite{Wa} for the affine case, and \cite[Proposition 2.1]{Hui} and 
\cite[Corollary 2.13]{Sa} for the projective case). We extend this fact by relaxing the hypotheses on $X$.

\begin{prop}\label{bd}
Let $X$ and $Y$ be projective integral $k$-schemes and let
$f:X\rightarrow Y$ be a morphism such that $\rr f_*\sO_X\simeq \sO_Y$. 
Then $f$ defines a morphism of functors
\begin{gather}\label{blowingdown}
f':{\rm Def}_X\longrightarrow {\rm Def}_Y
\end{gather}
where to a deformation $\mathcal X$ of $X$ over a local Artinian $k$-algebra $A$ with residue field $k$
associates the deformation $\mathcal{Y}:=(Y,f_*\sO_{\mathcal X})$ of $Y$ over $A$.
\end{prop}
\begin{proof}
Let $A$ be a local Artinian $k$-algebra as in the statement and let
\begin{eqnarray}\label{defg}
\centerline{ \xymatrix@=32pt{
X \ar[r] \ar[d]^{\bar g} & \mathcal{X} \ar[d]^{g}  \\
{\rm Spec}\,k \ar[r]^{\alpha} & {\rm Spec}\,A}} 
 \noindent
 \end{eqnarray}
be a deformation of $X$ over $A$. Moreover denote by $\bar h:Y\rightarrow {\rm Spec}\,k$ the 
structure morphism of $Y$. 
Then $\bar g=\bar h\circ f$ and $\mathcal{Y}$ admits a morphism $h$ to ${\rm Spec}\,A$ determined by the following morphism of $k$-algebras
$$A=\sO_{{\rm Spec}\,A}\longrightarrow \bar g_*\sO_{\mathcal X}=\bar h_*f_*\sO_{\mathcal X}=\bar h_*\sO_{\mathcal Y}.$$
We only need to prove that $h$ is a flat morphism. 
To this end we can suppose that $Y$ is affine.

\begin{claim}
 There is an isomorphism of complexes $\rr f_*\sO_{\mathcal X}\simeq \sO_{\mathcal{Y}}$.
\end{claim}
\begin{proof}
By the construction of $\mathcal{Y}$ we have $f_*\sO_{\mathcal{X}}\simeq \sO_{\mathcal{Y}}$, therefore we only need to prove the vanishings
$R^if_*\sO_{\mathcal{X}}=0$ for $i>0$.

Since $\rr f_*\sO_X\simeq \sO_Y$ we have $H^i(X,\sO_X)=0$ for all $i>0$.
This easily follows by taking cohomology from the following chain of isomorphisms (and from the fact that we are assuming $Y$ affine)
$$\rr \Gamma (X,\sO_X)\simeq \rr \Gamma (Y, \rr f_*\sO_X) \simeq \Gamma(Y,\rr f_*\sO_X).$$ 
Now as $g$ is a flat morphism we can apply the push-pull formula of \cite[Lemma 2.22 and Corollary 2.23]{Ku} to \eqref{defg} to have 
$$\rl \alpha^* \rr g_* \sO_{\mathcal{X}} \simeq \rr \Gamma (X,\sO_X).$$
Let $i_0:={\rm max}\, \{i\, | \, R^i g_*\sO_{\mathcal{X}}\neq 0\}$. If by contradiction $i_0>0$, then we would have 
$$0=L^{i_0}\alpha^*\rr g_*\sO_X \simeq \alpha^* R^{i_0}g_*\sO_{\mathcal X},$$ and therefore by Nakayama's lemma we would get the contradiction 
$R^{i_0} g_*\sO_{\mathcal{X}}=0$. We conclude that $i_0=0$ and moreover that 
$$\rr g_*\sO_{\mathcal X}\simeq \rr h_* \rr f_* \sO_{\mathcal{X}}\simeq \rr \Gamma (Y,\rr f_*\sO_{\mathcal{X}})\simeq \Gamma 
(Y,\rr f_* \sO_{\mathcal{X}}).$$
From this we see that for any $i>0$ we have  
$$0=R^i g_* \sO_{\mathcal{X}}\simeq \Gamma(Y,R^if_*\sO_{\mathcal{X}}).$$
Moreover, since the functor of global sections is exact on affine spaces, we finally get $R^i f_*\sO_{\mathcal{X}}=0$ for all $i>0$.
\end{proof}

In order to show that $h$ is flat, we will prove that for any coherent sheaf $\sF$ on ${\rm Spec}\,A$ we have $L^i h^*\sF=0$ for all $i<0$. 
But this follows from the projection formula (\cite[Proposition 5.6]{Ha1}) as the following chain of isomorphisms yields
\begin{eqnarray*}
 \rl h^*\sF & \simeq & \sO_{\mathcal{Y}} \stackrel{\rl}{\otimes} \rl h^*\sF\\
& \simeq & \rr f_*\sO_{\mathcal{X}}\stackrel{\rl}{\otimes} \rl h^*\sF\\
& \simeq & \rr f_*(\rl f^* \rl h^* \sF)\\
& \simeq & \rr f_* \rl g^*\sF\\
& \simeq & \rr f_*(g^*\sF).\\
\end{eqnarray*}
We conclude that $L^i h^* \sF=0$ for all $i<0$ since the derived push-forward of a sheaf lives in non-negative degrees.
\end{proof}

\begin{prop}\label{bdtanget}
The differential ${\rm d}f'$ to $f'$ in \eqref{blowingdown} can be described as the composition
  \begin{gather}\label{differential}
{\rm Ext}^1_{\sO_X}(\Omega_X,\sO_X)\longrightarrow {\rm Ext}^1_{\sO_X}(\mathbf{L}f^* \Omega_Y,\sO_X)\simeq {\rm Ext}^1_{\sO_Y}(\Omega_Y,\sO_Y)
\end{gather}
where the first map is obtained by applying the functor ${\rm Ext}_{\sO_X}^1(-,\sO_X)$ to 
the morphism $\bar \delta_1:\mathbf{L}f^* \Omega_Y\rightarrow \Omega_X$, and
the second by adjunction formula \cite[Corollary 5.11]{Ha1}. 
Moreover, if $f$ is birational and the exceptional locus of $f$ in $X$ has codimension at least $3$, then ${\rm d}f'$ is an isomorphism.
\end{prop}

\begin{proof}
We refer to \cite{Wa} for the description of ${\rm d}f'$ in the affine case. In the global case this is obtained as follows; we continue
using notation of Proposition \ref{bd} and its proof.
Let $\mathcal{X}$ be a first-order deformation of $X$ and let $\mathcal{Y}= (Y,f_*\sO_{\mathcal{X}})$ 
be the deformation determined by $f'$.
As $f_*\sO_{\mathcal{X}} = \sO_{\mathcal{Y}}$, we get 
a morphism $\widetilde{f}:\mathcal{X}\rightarrow \mathcal{Y}$ such that the top square of the diagram 
\eqref{definition} commutes. Therefore, as shown in the proof of Proposition \ref{tangentspacedef}, the right square of \eqref{diagrT1a} commutes 
and 
hence, under the morphism 
${\rm Ext}_{\sO_X}^1(\Omega_{X},\sO_{X})\stackrel{\bar \delta_1^*}{\longrightarrow} {\rm Ext}_{\sO_X}^1(\rl f^* \Omega_Y,\sO_X)$,
the distinguished triangle 
$$\sO_X\longrightarrow \Omega_{\mathcal{X}|X}\longrightarrow \Omega_X\longrightarrow \sO_X[1],$$ which determines the deformation 
$\mathcal{X}$,
is sent to the triangle 
$$\sO_X\longrightarrow \rl f^*\Omega_{\mathcal{Y}|Y}\longrightarrow \rl f^*\Omega_{Y}\longrightarrow \sO_X[1].$$ 
Consequently, via adjunction, this last triangle is sent to the exact sequence 
$0\rightarrow \sO_Y\rightarrow \Omega_{\mathcal{Y}|Y}\rightarrow \Omega_Y\rightarrow 0$ which is the sequence
determining the first-order deformation $\mathcal{Y}$ of $Y$. Therefore the morphism defined in \eqref{differential} takes $\mathcal{X}$ to 
$\mathcal{Y}$ which is what we needed to show.

For the second statement we complete $\bar \delta_1$ to a distinguished triangle:
$$\rl f^* \Omega_Y\longrightarrow \Omega_X \longrightarrow Q \longrightarrow \rl f^*\Omega_Y[1].$$ We note that since $f$ is an isomorphism
outside the exceptional locus, the supports of the cohomology sheaves $H^i(Q)$ of $Q$ have codimension $\geq 3$. Moreover, $H^i(Q)=0$ for $i\geq 1$ 
as $\rl f^*\Omega_Y$ lives in non-positive degrees. 
Then ${\rm Ext}_{\sO_X}^j(H^i(Q),\sO_X)=0$ for all $j\leq 2$ and all $i$, and therefore the spectral sequence
$E_2^{j,i}={\rm Ext}_{\sO_X}^j(H^{-i}(Q),\sO_X) \Rightarrow {\rm Ext}_{\sO_X}^{j+i}(Q,\sO_X)$ yields the vanishings $${\rm ker}\,{\rm d}f' \simeq
{\rm Ext}_{\sO_X}^1(Q,\sO_X)=0\quad \quad  \mbox{ and } \quad \quad {\rm coker}\, {\rm d}f' \simeq 
{\rm Ext}_{\sO_X}^2(Q,\sO_X)=0.$$
\end{proof}

The previous proposition leads to a criterion for the blowing-down morphism to be an isomorphism. This proves Criterion \ref{cri} of the 
Introduction.

\begin{cor}\label{corbd}
 Let $f:X\rightarrow Y$ be a birational morphism of integral projective $k$-schemes such that 
 the exceptional locus of $f$ is of codimension at least 
 three in $X$ and $\rr f_*\sO_X\simeq \sO_Y$. If $X$ is nonsingular, unobstructed (\emph{e.g.} $h^2(X,T_X)=0$), 
 and $h^0(X,T_X)=0$, then the blowing-down morphism $f':{\rm Def}_X\rightarrow {\rm Def}_Y$ is an isomorphism. 
\end{cor}

\begin{proof}
 By \cite[Corollary 2.6.4 and Corollary 2.4.7]{Se} the functor ${\rm Def}_X$ is prorepresentable 
 and smooth. Moreover, as $f$ is a small resolution, there is an isomorphism $f_*T_X\simeq T_Y$ whose proof can be found in 
 \cite[Lemma 21]{SV}.
 Then $H^0(Y,T_Y)\simeq H^0(X,T_X)=0$ and ${\rm Def}_Y$ is prorepresentable as well.
 At this point the corollary is a consequence of Proposition \ref{bdtanget} and the following criterion \cite[Remark 2.3.8]{Se}: if 
 $\gamma:G\rightarrow G'$ is a morphism of functors of Artin rings such that the differential 
 ${\rm d}\gamma$ is an isomorphism and both $G$ and $G'$ are prorepresentable with $G$
 smooth, then $\gamma$ is an isomorphism and $G'$ is smooth as well.
\end{proof}

In a similar fashion, we also show that it is possible to blow-down deformations of morphisms.

\begin{prop}\label{bdmor}
 Let $f:X\rightarrow Y$ be a resolution of singularities of 
a projective integral normal $k$-scheme $Y$ such that $\rr f_*\sO_X\simeq \sO_Y$. 
Moreover fix a smooth projective variety $Z$ together with two morphisms $f_1:X\rightarrow Z$ and $f_2:Y\rightarrow Z$ such that 
$f_1 = f_2 \circ f$.
Then $f$ defines a morphism of functors
$F:{\rm Def}_{f_1}\rightarrow {\rm Def}_{f_2}$ such that the following diagram 
\begin{equation}\label{commudiagrbdmor}
\centerline{ \xymatrix@=32pt{
 {\rm Def}_{f_1} \ar[r]^{p}\ar[d]^{F} & {\rm Def}_{X}\ar[d]^{f'}   \\
 {\rm Def}_{f_2} \ar[r]^{p'} & {\rm Def}_{Y}}} 
 \noindent
\end{equation}
commutes where $p$ and $p'$ are forgetful morphisms and $f'$ is the blowing-down morphism defined in Proposition \ref{bd}.
\end{prop}

\begin{proof}
By definition, a deformation $\widetilde{f}_1:\mathcal{X}\rightarrow \mathcal{Z}$ of $f_1$ over a local Artinian $k$-algebra $A$ with
residue field $k$ determines a deformation $\mathcal{X}$ of $X$ and a deformation 
$s:\mathcal{Z}\rightarrow {\rm Spec}\, A$ of $Z$. 
Furthermore, by Proposition \ref{bd}, $\mathcal{X}$ defines a deformation 
$\mathcal{Y}=(Y,f_*\sO_{\mathcal{X}})$ of $Y$.
As $\widetilde{f}_1$ 
determines a morphism of sheaves of 
$k$-algebras $f^*f_2^*\sO_{\mathcal{Z}}=f^*_1\sO_{\mathcal{Z}}\rightarrow \sO_{\mathcal{X}}$, by applying $f_*$ we get 
a morphism of sheaves of $k$-algebras 
$f_2^*\sO_{\mathcal{Z}}\rightarrow f_*\sO_{\mathcal{X}}\simeq \sO_{\mathcal{Y}}$, which in turn defines
a morphism $\widetilde{f}_2:\mathcal{Y}\rightarrow \mathcal{Z}$.
 
At this point in order to check that $\widetilde{f}_2$ is a deformation of $f_2$ we 
only need to show that the composition $\mathcal{Y}\stackrel{\widetilde{f}_2}{\rightarrow} \mathcal{Z}\stackrel{s}{\rightarrow} {\rm Spec}\, A$ 
is flat. 
This is equivalent to proving that for any coherent sheaf $\sF$ on ${\rm Spec}\,A$ the higher cohomology of $\rl (s\circ \widetilde{f}_2)^*\sF$ 
vanish.
But since $s\circ \widetilde{f}_1$ is flat, by projection formula we have that for any index $i<0$ 
(we use the symbol $H^i$ to denote the $i$-th cohomology of a complex):

\begin{eqnarray*}
 H^i(\rl(s\circ \widetilde{f}_2)^* \sF) & \simeq & H^i(\rl \widetilde{f}_2^* \rl s^* \sF)\\
& \simeq & H^i(\sO_{\mathcal{Y}}\stackrel{\rl}{\otimes}\rl \widetilde{f}_2^* \rl s^*\sF)\\
& \simeq & H^i(\rr f_{*}\sO_{\mathcal{X}}\stackrel{\rl}{\otimes} \rl \widetilde{f}_2^* \rl s^* \sF)\\
& \simeq & H^i(\rr f_{*}(\rl \widetilde{f}_1^* \rl s^*\sF))\\
& \simeq & H^i(\rr f_*(s\circ \widetilde{f}_1)^*\sF)=0.
\end{eqnarray*}

The commutativity of \eqref{commudiagrbdmor} follows from the definitions of all involved morphisms and functors.
\end{proof}

\section{Deformations of $W_d(C)$}\label{secdefWd}
  
Let $C$ be a complex smooth curve of genus $g\geq 3$. We denote by 
$$W_d(C)=\big\{[L]\in {\rm Pic}^d(C)\,|\, h^0(C,L)>0 \big\}$$ the Brill--Noether loci parameterizing
degree $d$ line bundles on $C$ having at least one non-zero global section. 
We recall that it is possible to put a scheme structure on $W_d(C)$ by means of Fitting ideals so that
$W_d(C)$ is an irreducible, normal, Cohen--Macaulay scheme of dimension $d$ (\cite[Corollary 4.5]{ACGH}). 
A resolution of singularities of $W_d(C)$ is provided by an Abel--Jacobi map 
\begin{gather*}
u_d:C_d\longrightarrow W_d(C),\quad  \quad P_1+\ldots +P_d\mapsto \sO_C \big(P_1+\ldots +P_d \big)
\end{gather*}
where $C_d$ is the $d$-fold symmetric product of $C$.
Note that a fiber of $u_d$ over
a point $[L]\in W_d(C)$ is nothing else than the linear series $|L|$ associated to $L$. 
Finally, by fundamental results of Kempf (\cite{Ke1}),
we have that $W_d(C)$ has at most rational singularities so that the following isomorphisms hold:
\begin{gather}\label{ratsing}
\rr u_{d*}\sO_{C_d}\simeq \sO_{W_d(C)} \quad \mbox{ and } \quad \rr u_{d*}\omega_{C_d}\simeq \omega_{W_d(C)}.
\end{gather}
By Proposition \ref{bd} there is then a well-defined blowing-down morphism  
$$u_d':{\rm Def}_{C_d}\longrightarrow {\rm Def}_{W_d(C)}.$$

The goal of this section is to prove the following:
\begin{theorem}\label{defWd}
 If $C$ is a smooth non-hyperelliptic curve of genus $g\geq 3$, 
 then the blowing-down morphism $u_d':{\rm Def}_{C_d}\rightarrow {\rm Def}_{W_d(C)}$ is an isomorphism of functors for all $1\leq d< g-1$.
\end{theorem}

The proof of the previous theorem requires a few technical results on the supports of higher-direct image sheaves 
of type $R^1 u_{d*}(u_d^*\Omega_{W_d(C)}\otimes \omega_{C_d})$ and 
$R^1u_{d*}(\Omega_{C_d / W_d(C)}\otimes \omega_{C_d})$.
We will collect these facts in the following subsection and we will show the proof of Theorem \ref{defWd}
in \S\ref{secproof}.

It is worth noticing that Fantechi (\cite{Fa}), by extending previous work of Kempf (\cite{Ke2}), proved the following:

\begin{theorem}[Fantechi]\label{fantechi}
 Let $C$ be a smooth curve of genus $g\geq 2$ and let $d\geq 2$ be an integer. 
 Then the quotient morphism $C^d\rightarrow C_d$ induces an isomorphism of functors of Artin rings ${\rm Def}_{C_d}\simeq {\rm Def}_C$ 
  if and only if $g\geq 3$.
\end{theorem}

Combining with Theorem \ref{defWd} we obtain 

\begin{cor}\label{corfantechi}
 If $C$ is a smooth non-hyperelliptic curve, then for all $1\leq d< g-1$ there are isomorphisms of functors 
 ${\rm Def}_C \simeq {\rm Def}_{W_d(C)}$. 
 \end{cor}

 It follows from the previous corollary that ${\rm Def}_{W_d(C)}$ is unobstructed and is prorepresented by a formal 
 power series in $3g-3$ variables as so is ${\rm Def}_C$ (\cite[Proposition 2.4.8 and Corollary 2.6.6]{Se}).

\subsection{Supports of special higher direct image sheaves}

We denote by 
$$W_d^i(C)=\big\{[L]\in {\rm Pic}^d(C)\,|\, h^0(C,L)\geq i+1\big\}$$ 
the Brill--Noether loci parameterizing degree $d$ line bundles on $C$ having 
at least $i+1$ linearly independent global sections, and by  
$$C^i_d=\big\{0\leq D\in {\rm Div}^d(C)\,|\, \dim |D|\geq i\big\}$$ 
the loci parameterizing degree $d$ effective divisors on $C$ whose associated linear series is of dimension at least $i$. 
Note that $u_d^{-1}(W_d^i(C))=C_d^i$.
We start by remarking a general fact concerning the supports of higher direct image sheaves under Abel--Jacobi maps. 
\begin{prop}\label{support}
 If $\sF$ is a coherent sheaf on $C_d$, then ${\rm supp}\, R^j u_{d*}\sF \subset W^j_d(C)$ for all $j>0$.
\end{prop}

\begin{proof}
There are fiber product diagrams

\begin{equation*}
\centerline{ \xymatrix@=32pt{
 C_d\backslash C^j_d \ar[r]^{\nu} \ar[d]^{\bar u_d} & C_d \ar[d]^{u_d} \\
 W_d(C)\backslash W_d^j(C)\ar[r]^{\mu} & W_d(C) \\ }}
\end{equation*} 

 \noindent where $\nu$ and $\mu$ are open immersions and $\bar u_d$ is the restriction of $u_d$ on $C_d\backslash C_d^j.$
Since $\dim u_d^{-1}([L])=\dim |L|< j$ for all 
$L\in W_d(C)\backslash W_d^j(C)$, we have $R^j\bar u_{d*}(\nu^*\sF)=0$ by \cite[Corollary 11.2]{Ha2}. Therefore by
base change we find $\mu^*R^j u_{d*}\sF=0$. 
\end{proof}

We now give more precise information regarding the supports of two specific higher direct image sheaves: 
$R^1 u_{d*}(u_d^*\Omega_{W_d(C)}\otimes \omega_{C_d})$ and $R^1u_{d*}(\Omega_{C_d / W_d(C)}\otimes \omega_{C_d})$ 
where $\Omega_{C_d / W_d(C)}$ is the sheaf of relative K\"{a}hler differentials. 
The main tool we use towards this study is Ein's cohomological computations of the dual of the normal bundle to the fibers of 
$u_d$ (\cite{Ein}).

\begin{prop}\label{R1a}
 For any $1\leq d\leq g-1$ the support of $R^1 u_{d*}(u_d^*\Omega_{W_d(C)}\otimes \omega_{C_d})$ is contained in $W_d^2(C)$.
\end{prop}

\begin{proof}
By Proposition \ref{support} we have that ${\rm supp}\,R^1u_{d*}(u_d^*\Omega_{W_d(C)}\otimes \omega_{C_d}) \subset W_d^1(C)$, 
so we only need to show that the stalks
$$R^1u_{d*}(u_d^*\Omega_{W_d(C)}\otimes \omega_{C_d})_{[L]}$$
 vanish for all $[L]\in W^1_d(C) \backslash W^2_d(C)$. From now on we fix an element $[L]\in W^1_d(C) \backslash W^2_d(C)$.
Recall that, as shown in \cite[Theorem 1.1]{Ein}, the normal bundle $N$ of the fiber $P_L:=u_d^{-1}([L])\simeq \rp ^1$ at $[L]$ sits in an exact sequence of the form
\begin{eqnarray}\label{normal}
0\longrightarrow N\longrightarrow H^1(C,\sO_C)\otimes \sO_{P_L}\longrightarrow H^1(C,L)\otimes \sO_{P_L}(1)\longrightarrow 0.
\end{eqnarray}
From this we easily deduce that ${\rm det}\, N^{\vee}\simeq \sO_{\rp^1}(g-d+1)$, and moreover 
that $\omega_{C_d|P_L}\simeq \omega_{P_L}\otimes {\rm det}\, N^{\vee}
\simeq \sO_{\rp^1}(g-d-1)$ by adjunction. 

We apply the theorem on formal functions to get the vanishing of the above mentioned stalks.
Denote by $\sI$ the ideal sheaf defining $E=E_1:=P_L$ in $C_d$ and let $E_n$ be the subscheme defined by $\sI^n$. 
We have exact sequences 
\begin{eqnarray}\label{seq2}
0\longrightarrow \sI^n / \sI^{n+1}\longrightarrow i_{(n+1)*} \sO_{E_{n+1}}\longrightarrow i_{n*} \sO_{E_n}\longrightarrow 0
\end{eqnarray}
where the maps $i_n:E_n\rightarrow C_d$ denote the natural inclusions.
 Set now $\sF=\sF_1:=u_d^*\Omega_{W_d(C)}\otimes \omega_{C_d}$ and 
$\sF_n:=i_n^*\sF=i_n^*(u_d^*\Omega_{W_d(C)}\otimes \omega_{C_d})$.
By the theorem on formal functions we obtain isomorphisms
$$R^1u_{d*}(u_d^*\Omega_{W_d(C)}\otimes \omega_{C_d})_{[L]}^{\widehat{}} \simeq \varprojlim H^1(E_n,\sF_n),$$
so that it is enough to check the vanishing of cohomology groups on the RHS.
By tensoring \eqref{seq2} by $\sF$, and by using the isomorphisms 
$\sI^n / \sI^{n+1}\simeq {\rm Sym}^n\, N^{\vee}$, we deduce new exact sequences
$$0\longrightarrow \sK_n \stackrel{\psi}{\longrightarrow} {\rm Sym}^n\, N^{\vee} \otimes \sF \longrightarrow
i_{(n+1)*}\sO_{E_{n+1}}\otimes \sF \longrightarrow i_{n*}\sO_{E_n}\otimes \sF \longrightarrow 0$$
where we denote by $\sK_n$ the kernel of $\psi$. 
We are interested in the vanishing of 
$H^1(E_{n+1},\sF_{n+1})\simeq H^1(C_d, i_{(n+1)*}\sO_{E_{n+1}}\otimes \sF)$. We proceed by induction on $n$.
The base step $n=0$ is easily proved as 
$$H^1(E,\sF_{|E})\simeq H^1(E, \omega_{C_d|E})^{\oplus d}
\simeq  H^1(\rp^1,\sO_{\rp^1}(g-d-1))^{\oplus d}=0.$$ Now we show that if $H^1(E_{n}, \sF_n)=0$, then also $H^1(E_{n+1},\sF_{n+1})=0$.
First of all we note that all we need is the vanishing of 
\begin{eqnarray}\label{van}
H^1(E, {\rm Sym}^n\, N^{\vee}\otimes \sF_{|E})\simeq H^1(\rp^1,{\rm Sym}^n N^{\vee}\otimes \sO_{\rp^1}(g-d-1))^{\oplus d}.
\end{eqnarray}
In fact, by denoting by $\sQ_n$ the cokernel of $\psi$, the inductive hypothesis tells us that 
$H^1(E_{n+1},\sF_{n+1})$ vanishes as soon as $H^1(E_{n+1}, \sQ_n)$ does. 
But this is the case
if $H^1(E,{\rm Sym}^n\, N^{\vee}\otimes \sF_{|E})=0$ as $H^2(E_{n+1},\sK_{n})=0$ (recall that $\dim E_{n+1}=1$).
 
Finally, in order to get the vanishing of the RHS of \eqref{van},
 we note that dualizing the sequence \eqref{normal} we get
 surjections ${\rm Sym}^n (H^1(C,\sO_C)^{\vee}\otimes 
 \sO_{\rp^1})\twoheadrightarrow {\rm Sym}^n N^{\vee}$ for all $n\geq 1$. Therefore there are surjections
 \begin{equation}\label{for1}
 H^1(\rp^1,{\rm Sym}^n \big(H^1(C,\sO_C)^{\vee}\otimes \sO_{\rp^1}\big)\otimes \sO_{\rp^1}(g-d-1))\twoheadrightarrow 
 H^1(\rp^1,{\rm Sym}^n N^{\vee}\otimes \sO_{\rp^1}(g-d-1))
 \end{equation}
 from which one easily deduces the vanishing of \eqref{van} as the LHS of \eqref{for1} is zero.
 \end{proof}

\begin{prop}\label{relativediff}
 For any $1\leq d< g-1$ the support of $R^1u_{d*}(\Omega_{C_d / W_d(C)}\otimes \omega_{C_d})$ is contained in $W_d^2(C)$.
\end{prop}
\begin{proof}
We follow the strategy of Proposition \ref{R1a} so that we only need to check the vanishings of
\begin{equation}\label{for2}
H^1\big(E,\big( \Omega_{C_d / W_d(C)}\big)_{|E}\otimes \omega_{C_d|E}\big)\quad  \mbox{ and }\quad 
H^1\big(E,{\rm Sym}^n N^{\vee}\otimes \big( \Omega_{C_d / W_d(C)}\big)_{|E}\otimes \omega_{C_d|E}\big)
\end{equation}
for all $n\geq 1$. 
We recall the isomorphism $\big(\Omega_{C_d / W_d(C)}\big)_{|E}\simeq \omega_E\simeq \sO_{\rp^1}(-2)$ 
(\emph{cf}. \emph{e.g.} \cite[Proposition 8.10]{Ha2}) 
and that $\omega_{C_d|E}\simeq \sO_{\rp^1}(g-d-1)$.
Therefore
$$H^1(E,\big( \Omega_{C_d / W_d(C)}\big)_{|E} \otimes \omega_{C_d|E})\simeq H^1(\rp ^1, \sO_{\rp^1}(g-d-3))=0$$ as soon as
 $d\leq g-2$. For the second set of groups in \eqref{for2} we note the isomorphisms 
 \begin{equation}\label{c1}
 H^1(E,{\rm Sym}^n N^{\vee}\otimes \big( \Omega_{C_d / W_d(C)}\big)_{|E} \otimes 
 \omega_{C_d|E})\simeq H^1(\rp^1,{\rm Sym}^n N^{\vee}\otimes \sO_{\rp^1}(g-d-3))
 \end{equation}
 and the surjections
 $$H^1(\rp^1,{\rm Sym}^n (H^1(C,\sO_C)^{\vee}\otimes \sO_{\rp^1})\otimes \sO_{\rp^1}(g-d-3))\twoheadrightarrow
 H^1(\rp^1,{\rm Sym}^n N^{\vee}\otimes \sO_{\rp^1}(g-d-3))$$ 
deduced from \eqref{normal}. As the groups on the LHS of the previous surjections vanish, so the groups in \eqref{c1} do. 
 \end{proof}

 \begin{rmk}
  The previous two propositions can be extended to all higher direct images to yield inclusions
  $${\rm supp}\, R^j u_{d*}(u_d^*\Omega_{W_d(C)}\otimes \omega_{C_d})\subset W^{j+1}_d(C)\quad  \mbox{ and }\quad 
  {\rm supp}\, R^j u_{d*}(\Omega_{C_d /W_d(C)}\otimes \omega_{C_d})\subset W^{j+1}_d(C)$$ for all $j\geq 1$ (the latter holds for $d\leq g-2$).
  While the proof of the first set of inclusions do not require any additional tools, for the latter we need to involve Bott's formula 
  to check that $H^j(\rp^j, \Omega_{\rp^j}\otimes \sO_{\rp^j}(g-d-1))=0$ for all $j\geq 1$.
 \end{rmk}

 \subsection{Proof of Theorem \ref{defWd}}\label{secproof}
 
 To prove Theorem \ref{defWd} we use the criterion \cite[Remark 2.3.8]{Se} which we have already recalled in Corollary \ref{corbd}.
 
 In our setting, the functor ${\rm Def}_{C_d}$ is prorepresentable and unobstructed by Theorem \ref{fantechi}. 
 Therefore we only need to prove that ${\rm Def}_{W_d(C)}$ is prorepresentable and that the differential to
 $u_d'$ is an isomorphism. 
 A sufficient condition for the prorepresentability of ${\rm Def}_{W_d(C)}$ 
 is the vanishing of $H^0(W_d(C),T_{W_d(C)})$ (\cite[Corollary 2.6.4]{Se}). 
 On the other hand, as $u_d$ is a small resolution (as we are supposing that $C$ is non-hyperelliptic), we obtain 
 an isomorphism $u_{d*}T_{C_d}\simeq T_{W_d(C)}$ (\cite[Lemma 21]{SV}) which immediately yields 
 $$H^0(W_d(C),T_{W_d(C)})\simeq H^0(C_d,T_{C_d})\simeq H^0(C^d,T_{C^d})^{\sigma_d}=0$$ by the K\"{u}nneth decomposition 
 (here $\sigma_d$ denotes the $d$-symmetric group).
    
 We now prove that the differential ${\rm d} u_d'$ is an isomorphism. This is slightly more difficult and it will take the rest 
 of the subsection. To begin with, we complete the morphisms 
$\zeta:\rl u_d^*\Omega_{W_d(C)}\rightarrow u^*_d\Omega_{W_d(C)}$ and
$\bar \delta_1: \rl u_d^*\Omega_{W_d(C)}  \rightarrow \Omega_{C_d}$ defined in \eqref{truncation} and \eqref{deltabar}  
to distinguished triangles:
\begin{gather}\notag
\rl u_d^*\Omega_{W_d(C)}\longrightarrow u^*_d\Omega_{W_d(C)}\longrightarrow M \longrightarrow \rl u_d^*\Omega_{W_d(C)}[1]\\ \notag
\rl u_d^*\Omega_{W_d(C)}  \longrightarrow \Omega_{C_d}\longrightarrow P\longrightarrow \rl u_d^*\Omega_{W_d(C)}  [1].
\end{gather}
Moreover, since $\bar \delta_1 = \delta_1\circ \zeta$ where $\delta_1:u_d^*\Omega_{W_d(C)}\rightarrow \Omega_{C_d}$ is the natural morphism, 
these triangles fit in the following commutative diagram:

\begin{align}\label{diagram}
\xymatrix@=32pt{
 \rl u_d^*\Omega_{W_d(C)}  \ar[r]^{\zeta}\ar[d]^{\bar \delta_1} & u_d^*\Omega_{W_d(C)} \ar[d]^{\delta_1}\ar[r] & M \ar[d]\ar[r] &  \rl u_d^*\Omega_{W_d(C)}  [1]\ar[d]\\
 \Omega_{C_d} \ar[r] \ar[d] & \Omega_{C_d} \ar[r]\ar[d]  & 0 \ar[d]\ar[r] & \Omega_{C_d} [1] \ar[d]\\ 
 P \ar[d]\ar[r] & N \ar[r]\ar[d] & M[1] \ar[r]\ar[d] & P[1]\ar[d]\\
  \rl u_d^*\Omega_{W_d(C)}  [1] \ar[r] & u_d^*\Omega_{W_d(C)}[1]  \ar[r] &  M[1]\ar[r] & \rl u_d^*\Omega_{W_d(C)}  [2] \\}
 \end{align}

 \noindent where $N$ is the cone of $\delta_1$.
Therefore, from the description of ${\rm d}u_d'$ as in Proposition \ref{bdtanget}, 
in order to prove that ${\rm d}u_d'$ is an isomorphism, it is enough to prove that 
${\rm Ext}^1_{\sO_{C_d}}(P,\sO_{C_d})={\rm Ext}^2_{\sO_{C_d}}(P,\sO_{C_d})=0$ which is implied by the vanishings
$${\rm Ext}^i_{\sO_{C_d}}(N,\sO_{C_d})=0\quad \mbox{ and }\quad {\rm Ext}^j_{\sO_{C_d}}(M[1],\sO_{C_d})=0 \quad \mbox{ for }i=1,2\quad \mbox{and}\quad  j=2,3.$$
The key ingredient to prove these vanishings is the Grothendieck--Verdier duality which 
reduces calculations from $C_d$ to the Jacobian $J(C)$ of $C$. 
We fix a point $Q\in C$ and denote by $\iota_d$ the closed immersion 
$$\iota_d: W_d(C) \hookrightarrow J(C)\quad \quad L\mapsto L\otimes \sO_C(-dQ).$$
Moreover we define the composition 
\begin{eqnarray}\label{fd}
f_d:=\iota_d \circ u_d:C_d\longrightarrow J(C), \quad \quad P_1+\cdots +P_d\mapsto \sO_{C}(P_1+\cdots +P_d-d\, Q).
\end{eqnarray}
Then applications of Grothendieck--Verdier duality (\cite[p. 7-8]{Ha1}) yield isomorphisms
\begin{equation*}
{\rm Ext}_{\sO_{C_d}}^j(N,\sO_{C_d})\simeq{\rm Ext}_{\sO_{J(C)}}^{j+g-d}(\rr f_{d*}(N\otimes \omega_{C_d}),\sO_{J(C)})
\end{equation*}
and 
\begin{equation*}
{\rm Ext}^j_{\sO_{C_d}}(M[1],\sO_{C_d})\simeq {\rm Ext}_{\sO_{J(C)}}^{j+g-d}(\rr f_{d*}(M[1]\otimes \omega_{C_d}),\sO_{J(C)}).
\end{equation*}
At this point the proof that ${\rm d}u_d'$ is an isomorphism follows from the following Propositions \ref{Ext3} and \ref{N}.

\begin{prop}\label{Ext3}
If $C$ is a smooth non-hyperelliptic curve of genus $g$, then for $j=2,3$ we have
 \begin{align*}
 {\rm Ext}_{\sO_{J(C)}}^{j+g-d}(\rr f_{d*}(M[1]\otimes \omega_{C_d}),\sO_{J(C)})=0.
 \end{align*}
\end{prop}

\begin{prop}\label{N}
 If $C$ is a smooth non-hyperelliptic curve of genus $g$, then for $j=1,2$ we have
 \begin{align*}
 {\rm Ext}_{\sO_{J(C)}}^{j+g-d}(\rr f_{d*}(N\otimes \omega_{C_d}),\sO_{J(C)})=0.
\end{align*}
 \end{prop}

\begin{proof}[Proof of Proposition \ref{Ext3}]
Our assertion is equivalent to ${\rm Ext}_{\sO_{J(C)}}^{j+g-d}(\rr f_{d*}(M\otimes \omega_{C_d}),\sO_{J(C)})=0$ for $j=1,2$.
First of all we note that 
\begin{eqnarray}\label{R10}
{\rm supp}\,R^j f_{d*}(M\otimes \omega_{C_d})\subset W^1_d(C)\quad \mbox { for }\quad j=-1,0.
\end{eqnarray}
To see this we first tensorize the top distinguished triangle of \eqref{diagram} by $\omega_{C_d}$, 
and then we apply the functor $\rr u_{d*}$. Hence 
projection formula (\cite[Proposition 5.6]{Ha1}), together with the isomorphism \eqref{ratsing}, yields an exact sequence 
$$0\longrightarrow R^{-1}u_{d*}(M\otimes \omega_{C_d})\longrightarrow \Omega_{W_d(C)}\otimes \omega_{W_d(C)}\stackrel{\kappa}{\longrightarrow} 
u_{d*}(u^*_d\Omega_{W_d(C)}\otimes \omega_{C_d})\longrightarrow$$ 
$$u_{d*}(M\otimes \omega_{C_d})\longrightarrow R^1 u_{d*}(u_d^*\Omega_{W_d(C)}\otimes \omega_{C_d})$$ such that $\kappa$ is an isomorphism outside 
the singular locus of $W_d(C)$, i.e. $W^1_d(C)$. This says that $R^{-1}u_{d*}(M\otimes \omega_{C_d})$ is supported on $W_d^1(C)$ and moreover, since 
$R^1u_{d*}(u_d^*\Omega_{W_d(C)}\otimes \omega_{C_d})$ is supported on $W_d^2(C)$ by Proposition \ref{R1a}, we find that $u_{d*}(M\otimes \omega_{C_d})$ is 
supported on $W_d^1(C)$ as well. Finally the statement in \eqref{R10} follows as $\iota_d$ is a closed immersion.

We now point out that, for all $j\geq 1$, there are isomorphisms 
$$R^j u_{d*}(M\otimes \omega_{C_d})\simeq R^j u_{d*}(u_d^*\Omega_{W_d(C)}\otimes \omega_{C_d})$$ deduced 
from diagram \eqref{diagram}. These isomorphisms, together with
Propositions \ref{R1a} and \ref{support}, yield
\begin{gather}\label{finally}
{\rm supp}\, R^1 f_{d*}(M\otimes \omega_{C_d})\subset W_d^2(C)\quad \mbox{ and }\quad 
{\rm supp}\, R^j f_{d*}(M\otimes \omega_{C_d})\subset W_d^j(C)\quad \mbox{ for all }\quad j\geq 2.
\end{gather}
  
Consider now the spectral sequence (\cite[p. 58]{Huy})
$$E_2^{p,q}={\rm Ext}_{\sO_{J(C)}}^p (R^{-q}f_{d*}(M\otimes \omega_{C_d}),\sO_{J(C)})\Rightarrow {\rm Ext}_{\sO_{J(C)}}^{p+q}(\rr f_{d*}(M\otimes \omega_{C_d}),\sO_{J(C)}).$$
We are interested in the vanishing of the terms on the lines $p+q=1+g-d$ and $p+q=2+g-d$, and therefore in the vanishing of the terms 
$E^{1+g-d-q,q}_2$ and $E^{2+g-d-q,q}_2$ for $q\leq 1$.
However this easily follows from the general fact that ${\rm Ext}^k_{\sO_{J(C)}}(\sF,\sO_{J(C)})=0$ 
for any coherent sheaf $\sF$ on $J(C)$ such that $\codim {\rm supp\,}\sF>k$, and from Martens' Theorem saying that $\dim W_d^j(C)\leq d-2j-1$ if 
$C$ is non-hyperelliptic (\cite[Theorem 5.1]{ACGH}). In fact, in this way, we obtain $E_{2}^{p,q}=0$ for couples $(p,q)$ such that either 
$p\leq g-d-2q$ and $q\leq -2$, or $p\leq 4+g-d$ and $q=-1$, or $p\leq 3+g-d$ and $q=-1,0$. 
\end{proof}

\begin{proof}[Proof of Proposition \ref{N}]
We consider the spectral sequence 
$$E_2^{p,q}=R^p u_{d*}(H^q(N\otimes \omega_{C_d}))\Rightarrow R^{p+q}u_{d*}(N\otimes \omega_{C_d})$$
in order to compute the supports of the sheaves $\sF^j:= R^j f_{d*}(N\otimes \omega_{C_d})$.
By noting that 
$$H^0(N\otimes \omega_{C_d})={\rm Ker}\,(\delta_1\otimes \omega_{C_d}), \quad H^{-1}(N\otimes 
\omega_{C_d})=\Omega_{C_d / W_d(C)}\otimes \omega_{C_d},\quad  
\mbox{and}\quad H^j(N\otimes \omega_{C_d})=0\quad  \mbox{else},$$   
by Propositions \ref{support} and \ref{relativediff} we find
\begin{gather*}
{\rm supp}\,\sF^{-1}\subset W^1_d(C), \quad {\rm supp}\,\sF^0\subset W_d^1(C),\quad  {\rm supp}\,\sF^1\subset W^2_d(C) \quad \mbox{ and } \\\notag
{\rm supp} \,\sF^j \subset W^j_d(C)\quad \mbox{ for all } \quad j\geq 2\notag.
\end{gather*} 
At this point to compute the groups ${\rm Ext}^{j+g-d}(\rr f_{d*}(N\otimes \omega_{C_d}),\sO_{J(C)})$ for $j=1,2$
we use the spectral sequence 
$$E_2^{p,q}={\rm Ext}_{\sO_{J(C)}}^p(\sF^{-q},\sO_{J(C)})\Rightarrow  {\rm Ext}_{\sO_{J(C)}}^{p+q}(\rr f_{d*}(N\otimes \omega_{C_d}),
\sO_{J(C)})$$ and we argue as in Proposition \ref{Ext3}.
\end{proof}

\section{Simultaneous deformations of $W_d(C)$ and $J(C)$}\label{secapp}
  
In this section we aim to prove Theorem \ref{intrdefiota}. We start by proving some general facts regarding the closed immersion
$\iota_d:W_d(C)\hookrightarrow J(C)$.
 \begin{prop}\label{albanese}
  Let $C$ be a smooth curve of genus $g\geq 2$ and let $1\leq d\leq g$ be an integer. 
  Then
  \begin{enumerate}
  \item[i).] $H^j(C_d,\sO_{C_d})\simeq \wedge ^j H^1(C,\sO_C)$ for all $j\leq d$.\\
  \item[ii).] $H^j(J(C),\sI_{W_d(C)})=0$ for all $j\leq d$.\\
  \item[iii).] ${\rm Ext}_{\sO_{J(C)}}^j(\Omega_{J(C)},\sO_{J(C)})\simeq 
  {\rm Ext}_{\sO_{J(C)}}^j(\Omega_{J(C)},\iota_{d*}\sO_{W_d(C)})$ for all $j\leq d$.
\end{enumerate}
  \end{prop}
  
\begin{proof}
The proof of $(i)$ can be found in \cite{Mac}. Nonetheless we present here a proof for the reader's ease. 
 Denote by $\pi_d:C^d\rightarrow C_d$ the quotient morphism realizing the symmetric product $C_d$ as quotient of the $d$-fold product
 $C^d$ under the action of the symmetric group $\sigma_d$.
 Therefore K\"{u}nneth's decomposition yields
 \begin{eqnarray*}
 H^*(C_d,\sO_{C_d})\simeq H^*\big(C^d,\big(\pi_{d*}\sO_{C^d}\big)^{\sigma_d}\big) & \simeq &
 H^*(C^d,\sO_{C^d})^{\sigma_d}\\
 & \simeq & \big(H^*(C,\sO_C)^{\otimes d}\big)^{\sigma_d}\\
 & \simeq & \bigoplus_{j=0}^d \big( {\rm Sym}^{d-j} H^0(C,\sO_C)\otimes \wedge^{j} H^1(C,\sO_C)\big)[-j]\\
 & \simeq & \bigoplus_{j=0}^d \wedge^j H^1(C,\sO_C)[-j].
 \end{eqnarray*}

 Now we turn to the proof of $(ii)$.
 By recalling the definition of $f_d = \iota_d \circ u_d :C_d\rightarrow J(C)$ in \eqref{fd}, we get isomorphisms 
 $$f_d^*H^{j}(J(C),\sO_{J(C)})\simeq 
 \wedge^j f_d^* H^1(J(C),\sO_{J(C)})\simeq \wedge^j H^1(C,\sO_{C})\simeq H^j(C_d,\sO_{C_d})$$
 thanks to the universal property of $J(C)$ and by $(i)$. 
 Moreover, since $\rr u_{d*}\sO_{C_d}\simeq \sO_{W_d(C)}$, we obtain isomorphisms 
 $$u_d^*H^j(W_d(C),\sO_{W_d(C)})\simeq H^j(C_d,\sO_{C_d})\simeq u_d^* \iota_d^* H^j(J(C),\sO_{J(C)}).$$ 
  These immediately yield $(ii)$ once one looks 
 at the long exact sequence in cohomology induced by the short exact sequence 
 \begin{equation}\label{definingWd}
 0\longrightarrow \sI_{W_d(C)}\longrightarrow \sO_{J(C)}\longrightarrow \iota_{d*}\sO_{W_d(C)}\longrightarrow 0. 
 \end{equation}

  Finally, for the last point it is enough to apply $\rr {\rm Hom}_{\sO_{J(C)}}(\Omega_{J(C)},-)$ 
  to the sequence \eqref{definingWd} and to use $(ii)$. 
 This yields the claimed isomorphisms for $j\leq d-1$ together with
 an injection ${\rm Ext}^d_{\sO_{J(C)}}(\Omega_{J(C)},\sO_{J(C)})\hookrightarrow {\rm Ext}_{\sO_{J(C)}}^d(\Omega_{J(C)},\iota_{d*}\sO_{W_d(C)})$ 
 which is an isomorphism 
for dimensional reasons.
\end{proof}

The deformations of Abel--Jacobi maps $f_d:C_d\rightarrow J(C)$ have been studied by Kempf. In particular, in \cite{Ke2}, Kempf shows that 
these deformations are all induced by those of
$C_d$ in case $C$ is non-hyperelliptic. However, in view of Theorem \ref{fantechi}, Kempf's result extends to all smooth curves 
of genus $g\geq 3$. In the following proposition we present a slightly different proof of this fact in the case $d\leq g$ 
by means of the theory of 
deformations of holomorphic maps developed by Namba in \cite{Na}. Moreover, we include a statement regarding the
deformations of the closed immersion $\iota_d:W_d(C)\hookrightarrow J(C)$. 
In combination with Theorem \ref{intrdefW_d}, this in particular proves Theorem \ref{intrdefiota} of the Introduction.

\begin{prop}\label{diffprop}
If $C$ is a smooth curve of genus $g\geq 2$, then for all $1\leq d \leq g$ 
the forgetful morphism ${\rm Def}_{f_d}\rightarrow {\rm Def}_{C_d}$ is an isomorphism. 
Moreover, if in addition $C$ is non-hyperelliptic, 
then the forgetful morphism ${\rm Def}_{\iota_d}\rightarrow {\rm Def}_{W_d(C)}$ is an isomorphism for all $1\leq d < g-1$.
\end{prop}

\begin{proof}
We start with the proof that the forgetful morphism $p:{\rm Def}_{f_d}\rightarrow {\rm Def}_{C_d}$ is an isomorphism. 
First of all we show that the differential ${\rm d}p$ is an isomorphism.
Since the varieties $C_d$ and $J(C)$ are smooth, the tangent space $T_{f_d}^1$ and the obstruction space $T_{f_d}^2$ 
to ${\rm Def}_{f_d}$ fit in an exact sequence (\cite[\S3.6]{Na}):
\begin{gather*}
\ldots \longrightarrow H^0(C_d,T_{C_d})\oplus H^0(J(C),T_{J(C)})\stackrel{\lambda_0}{\longrightarrow} 
H^0(C_d, f_d^*T_{J(C)})\longrightarrow T^1_{f_d}\\ 
\longrightarrow 
H^1(C_d,T_{C_d})\oplus H^1(J(C),T_{J(C)})\stackrel{\lambda_1}{\longrightarrow} H^1(C_d, f_d^*T_{J(C)})\longrightarrow T^2_{f_d}\\
\longrightarrow 
H^2(C_d,T_{C_d})\oplus H^2(J(C),T_{J(C)})\stackrel{\lambda_2}{\longrightarrow} H^2(C_d, f_d^*T_{J(C)})\longrightarrow \ldots
\end{gather*}
Moreover, the map $T^1_{f_d}\rightarrow H^1(C_d,T_{C_d})$ is identified to the differential ${\rm d}p$ and 
$T^2_{f_d}\rightarrow H^2(C_d,T_{C_d})$ is an obstruction map for ${\rm Def}_{f_d}$.

By the isomorphism $\rr f_{d*}\sO_{C_d}\simeq \iota_{d*}\sO_{W_d(C)}$ and Proposition \ref{albanese}, we see that the morphisms 
$\lambda_0$, $\lambda_1$ and $\lambda_2$ 
induce isomorphisms $$H^i(J(C),T_{J(C)})\simeq  H^i(C_d, f_d^*T_{J(C)})\quad  \mbox{for}\quad i=0,1,2.$$
Hence the differential ${\rm d}p:T^1_{f_d}\rightarrow T^1_{C_d}$ is an isomorphism and 
${\rm Def}_{f_d}$ is less obstructed than ${\rm Def}_{C_d}$. Moreover, by \cite[Proposition 2.3.6]{Se}, we conclude that 
${\rm Def}_{f_d}$ is unobstructed since ${\rm Def}_{C_d}$ is so.
Finally, according to the criterion \cite[Remark 2.3.8]{Se}, $p$ is an isomorphism as soon as  
both ${\rm Def}_{f_d}$ and ${\rm Def}_{C_d}$ are prorepresentable. 
But these facts follow from the general criterion of prorepresentability
\cite[Proposition 13]{Ma} and the vanishing $H^0(C_d,T_{C_d})=0$.

We now prove the second statement and suppose that $C$ is non-hyperelliptic.
By Proposition \ref{bdmor} there is a blowing-down morphism $t:{\rm Def}_{f_d}\rightarrow {\rm Def}_{\iota_d}$ 
which fits into a commutative diagram

\begin{equation*}
\centerline{ \xymatrix@=32pt{
 {\rm Def}_{f_d} \ar[r]^{p}\ar[d]^{t} & {\rm Def}_{C_d}\ar[d]^{u_d'}   \\
 {\rm Def}_{\iota_d} \ar[r]^{p'} & {\rm Def}_{W_d(C)}  }} 
\end{equation*}

\noindent where $p'$ is the forgetful morphism.
We notice that the differential ${\rm d}p'$ is nothing else than the morphism $T^1_{\iota_d}\rightarrow T^1_{W_d}$ of 
the exact sequence \eqref{exactseq} associated to $f=\iota_d$ (\emph{cf}. also Remark \ref{bdtanget}). 
Therefore, since by Proposition \ref{albanese} the second components of $\lambda_0$ and $\lambda_2$ of the same sequence are isomorphisms, 
we get that 
${\rm d}p'$ is an isomorphism. Therefore ${\rm d}t$ is an isomorphism too since we have already shown that 
both ${\rm d}p$ and ${\rm d}u_d'$ are isomorphisms (Theorem \ref{defWd}). Moreover,
${\rm Def}_{\iota_d}$ is prorepresentable by \cite[Proposition 13]{Ma}. 
Therefore, since ${\rm Def}_{f_d}$ is unobstructed, we have that $t$ is an isomorphism
(\cite[Remark 2.3.8]{Se}). Finally, as both $p$ and $\widetilde{u}_d$ are isomorphisms, this yields that $p'$ is an isomorphism as well.
\end{proof}

\section{Deformations of Fano surfaces of lines}\label{secfano}

In this section we prove some deformation-theoretic statements regarding the Fano surface of lines.
Let $Y\subset \rp^4$ be a smooth cubic hypersurface and let $F(Y)$ be the Fano scheme parameterizing lines on $Y$. Then $F(Y)$ is a smooth 
surface which embeds in the intermediate Jacobian $J(Y)$.
We denote by $\iota:F(Y)\hookrightarrow J(Y)$ this embedding and we recall that 
 the tangent space to ${\rm Def}_{F(Y)}$ has dimension $10$ while
the obstruction space dimension $40$. 
\begin{prop}
 The surface $F(Y)$ is co-stable in $J(Y)$, \emph{i.e.} the forgetful morphism 
 $p_{F(Y)}:{\rm Def}_{\iota}\rightarrow {\rm Def}_{F(Y)}$ is smooth. Moreover, ${\rm d}p_{F(Y)}$ is an isomorphism and ${\rm Def}_{\iota}$ is 
 less obstructed than ${\rm Def}_{F(Y)}$.
 \end{prop}
 \begin{proof}
 By \cite[Theorem 11.19]{CG} (\emph{cf}. \cite[Theorem 4.1]{LT} for a different proof)
 there is an isomorphism $J(Y)\simeq {\rm Alb}(F(Y))$ between the intermediate Jacobian $J(Y)$ and the Albanese variety of $F(Y)$
 from which we get an isomorphism $H^1\big(F(Y),\sO_{F(Y)})\simeq H^1(J(Y),\sO_{J(Y)}\big)$. Moreover,
 by \cite[(12.1)]{CG} we deduce a further isomorphism 
 \begin{eqnarray*}
  H^2(F(Y),\sO_{F(Y)}) & \simeq & H^1(F(Y),\sO_{F(Y)})\wedge H^1(F(Y),\sO_{F(Y)})\\
 & \simeq & H^1(J(Y),\sO_{J(Y)})\wedge H^1(J(Y),\sO_{J(Y)}) \\ 
 & \simeq & H^2(J(Y),\sO_{J(Y)}).
 \end{eqnarray*}
 
 Therefore, 
 by looking at the long exact sequence in cohomology induced by $0\rightarrow \sI_{F(Y)}\rightarrow \sO_{J(Y)}\rightarrow \sO_{F(Y)}\rightarrow 0$, 
 we deduce that $$H^i\big(J(Y),\sI_{F(Y)}\otimes T_{J(Y)}\big) = 0\quad  \mbox{ for }\quad i=0,1,2.$$
But the vanishing of these groups for $i=2$ is a sufficient condition for co-stability (\cite[Proposition 3.4.23]{Se}). 
On the other hand, the vanishing of all of them allow us to reason as in Proposition \ref{diffprop} to obtain the other statements.
  \end{proof} 

\end{document}